\documentclass[11pt]{amsart}
\usepackage{amsmath}
\usepackage{amsthm}
\usepackage{amsfonts}
\newtheorem{theorem}{Theorem}[section]

\newtheorem{lemma}[theorem]{Lemma}
\newtheorem{proposition}[theorem]{Proposition}

\theoremstyle{definition}
\newtheorem{definition}[theorem]{Definition}
\theoremstyle{remark}
\newtheorem{remark}[theorem]{Remark}
\numberwithin{equation}{section}

\begin{document}

\title[reflection of free boundary minimal surfaces]
{Free boundary minimal surfaces and the reflection principle}
\author[ J. CHOE]{JAIGYOUNG CHOE}
\thanks{Supported in part by NRF-RS-2023-00246133}

\address{Korea Institute for Advanced Study, Seoul, 02455, Korea}
\email{choe@kias.re.kr}

\begin{abstract} We show that a minimal surface meeting a sphere at a $90^\circ$ angle can be reflected across the sphere. Using this reflection,  we prove the uniqueness that every embedded free boundary minimal annulus in a ball is necessarily the critical catenoid.
\\

\end{abstract}

\maketitle

\section{introduction}
In 1873, Schwarz \cite{Sc} found a way to extend the domain of the definition of a complex analytic function. It states that if an analytic function $f(z)$ is defined in the upper half-plane, extends to a continuous function on the real axis, and takes on real values on the real axis, then one can extend it to an analytic function in the whole plane by the formula
$$f(\overline{z})=\overline{f(z)}.$$
What is remarkable in this extension is that the resulting function must also be analytic along the real axis, even though one assumes no differentiability there.

The reflection principle can be used to reflect a harmonic function $h(x,y)$ defined in the upper half-plane, which continuously extends to the zero value on the $x$-axis. The extension of $h$ to the lower half-plane is based on the rule $$ h (x,-y) = -h (x, y).$$

Schwarz's reflection principle  has a natural generalization for minimal surfaces as well: If a minimal surface $\Sigma$ contains a line segment $\ell$ on its boundary, then $\Sigma$ can be analytically extended across $\ell$ by rotating $\Sigma$ about $\ell$ by $180^\circ$. Moreover, if $\Sigma$ is perpendicular to a plane $\Pi$ along $\partial\Sigma\cap\Pi$, one can extend $\Sigma$
across $\partial\Sigma\cap\Pi$ by taking its mirror image over $\Pi$.

In \cite{C} the author extended Schwarz's reflection: If $\Sigma$ meets $\Pi$ along $\partial\Sigma\cap\Pi$ at a constant contact angle$(\neq90^\circ)$, then $\Sigma$ has an analytic reflection $\Sigma^*$ across $\Pi$, so that $\Sigma\cup\Sigma^*$ is minimal.

In this paper we further generalize the reflection principle: The sphere can become a mirror like the plane. To be precise, one can reflect a minimal surface  across a sphere if it is perpendicular to the sphere along its boundary.

This new reflection principle will be pivotal in studying free boundary minimal surfaces in a ball. Nitsche proved that every free boundary minimal disk in a ball must be a flat equatorial disk \cite{N}. Nitsche's result has led many mathematicians, including Nitsche himself and Fraser-Li \cite{FN}, to conceive the conjecture that every embedded free boundary minimal annulus in a ball should be the critical catenoid, the part of the catenoid in a ball perpendicular to the boundary sphere. In this conjecture, the hypotheses of embeddedness and $90^\circ$-contact-angle are indispensable, as Fern\'{a}ndez-Hauswirth-Mira have recently constructed immersed free boundary minimal annuli and embedded minimal annuli with non-orthogonal contact angle in a ball \cite{FHM}. Recently, there have been partial answers to the conjecture: \cite{KM}, \cite{LY}, \cite{M}, \cite{S}. In this paper, we use the reflection principle to prove this conjecture in the affirmative as follows.

Solve the Cauchy problem for the Laplacian to choose specific isothermal coordinates. Use these coordinates to transform the Steklov condition into the Schwarz condition and establish the reflection principle for the free boundary minimal surfaces in a ball. Reflect the free boundary minimal annulus $\Sigma$ across its two boundaries alternatingly and infinitely many times. Then, one can get a complete minimal surface with two ends $\widetilde{\Sigma}$. If $\Sigma$ is embedded, $\widetilde{\Sigma}$ has a total curvature of $-4\pi$. Therefore $\widetilde{\Sigma}$ must be the catenoid, and $\Sigma$ the critical catenoid.

{\it Acknowledgments.} The author thanks Pablo Mira for helpful discussions on the Cauchy problem.

\section{Cauchy problem}
Given a  minimal surface $\Sigma$ in $\mathbb R^3$, one can introduce isothermal coordinates $x,y$ to express the metric of $\Sigma$  as
$$ds^2=F(x,y)^2(dx^2+dy^2).$$ 
Let $\Psi$
be the conformal harmonic map $\Psi$ from $D\subset\mathbb R^2$ onto $\Sigma$ that pushes forward the Euclidean coordinates of $\mathbb R^2$ to the isothermal coordinates $x,y$ on $\Sigma$. Suppose $\Sigma$ is simply connected or doubly connected, i.e., annular. So let $D=\{(x,y):x^2+y^2<1,\,0<y\}$ and $\delta=\{(x,0):-1<x<1\}$, or let $D$ be an infinite strip $\mathbb R\times (0,a)$ with $\delta:=\mathbb R\times\{0\}$, and $\Psi$ a periodic map.  This paper concerns the case where $\Sigma$ has a free boundary $\gamma\subset\partial\Sigma$ in the unit sphere $\mathbb S^2$. One can choose $\Psi$ so that $\gamma=\Psi(\delta)$.  By Lewy's regularity theorem \cite{L}, $\gamma$ is an analytic curve and $F(x,0)$ is an analytic positive function on $\gamma$. 

To establish the reflection principle for $\Sigma$, we need specific isothermal coordinates $x,y$, which satisfies $F(x,y)=1$ along the free boundary $\gamma$. To find such $x,y$, we first need to solve the Cauchy problem for the Laplacian.
\begin{enumerate}
\item[] {\bf Cauchy problem}. Given analytic functions $g(x)$ and $f(x)$ on the free boundary $\gamma\subset\partial\Sigma\subset\mathbb S^2$ of the minimal surface  $\Sigma\subset\mathbb R^3$, find a harmonic function $h(x,y)$ on $\Sigma$ satisfying
\begin{equation*}\label{cc}
h=g\,\,\,\, {\rm and}\,\,\,\, \frac{\partial h}{\partial \nu}=f\,\,\,\, {\rm along}\,\,\,\,\gamma,
\end{equation*}
where $\nu$ is the inward unit conormal to $\gamma$ on $\Sigma$.
\end{enumerate}
It is only known that the Cauchy problem for the Laplacian in $\Sigma$ is solvable in a local neighborhood of $\gamma$ \cite{HV}. Moreover, it is ill-posed because even a slight variation in the initial data along $\gamma$ can significantly change the solution function. Therefore we solve the problem only for the simplest case: 
$$g\equiv0,\, \,\,{\rm and}\,\,\,f\equiv1 \,\,\,{\rm along\,}\,\,\gamma$$
on $\Sigma$ in $\mathbb R^3$.

\begin{lemma}\label{sub}
Let's define $\Sigma,x,y,F(x,y),\Psi,D$, $\delta$, $\gamma$ and $\nu$ as above.
Then there exists a superharmonic function $k$ on $\Sigma$ satisfying the Cauchy conditions
\begin{equation}\label{cc1}
k=0,\,\,\,\,\frac{\partial k}{\partial \nu}=1\,\,\,{\rm along}\,\,\,\gamma.
\end{equation}
\end{lemma}

\begin{proof}
Let $x_1,x_2,x_3$ be the rectangular coordinates of $\mathbb R^3$ such that $(x_1,x_2,x_3)=(0,0,0)$ at the center of the unit ball $B$. $x_1,x_2,x_3$ are harmonic on the minimal surface $\Sigma$. Define
$$r^2=x_1^2+x_2^2+x_3^2\,.$$
Then 
\begin{eqnarray*}
\Delta r^2&=&2\sum_{i=1}^3(x_i\Delta x_i+|\nabla x_i|^2)\\
&=&4,
\end{eqnarray*}
and
\begin{eqnarray*}
\Delta\log r&=&{\rm div}\left(\frac{\nabla r^2}{2r^2}\right)\\
&=&-\frac{2}{r^2}|\nabla r|^2+\frac{2}{r^2}\\
&\geq&0.
\end{eqnarray*}
Hence $$k:=-\log r$$ is a superharmonic function on $\Sigma$ satisfying the Cauchy conditions (\ref{cc1}).
\end{proof}

\begin{lemma}\label{har1}
Under the same hypotheses as Lemma \ref{sub}, there exists a harmonic function $h$ on $\Sigma$ satisfying the Cauchy conditions
\begin{equation}\label{cc2}
h=0,\,\,\,\,\frac{\partial h}{\partial \nu}=1\,\,\,{\rm along}\,\,\,\gamma.
\end{equation}
\end{lemma}
\begin{proof}
Consider the following PDE on $\Sigma$,
\begin{equation}\label{pde}
{L}u:=\Delta u+\left(-\frac{2}{r^2}|\nabla r|^2+\frac{2}{r^2}\right)=0.
\end{equation}
Lemma \ref{cc1} implies $u=k$ is a solution of $Lu=0$, and so the following are equivalent:
\begin{center}
$u$ is a solution of $Lu=0.\,\,\Longleftrightarrow\,\,u-k$ is harmonic.
\end{center}
Moreover, the boundary value problem for ${L}u=0$ has a unique solution $u$ as the sum of $k$ and $h$, where $h$ satisfies
$$\Delta h=0\,\,\,{\rm in}\,\,\,\Sigma,\,\,\,\,h=u-k\,\,\,{\rm on}\,\,\,\partial\Sigma.$$
Let $\mathcal{H}^\gamma$ be the set of all solutions of (\ref{pde}) on $\Sigma$  which vanish along $\gamma$ and are continuous on $\partial\Sigma$. Define a subset $\mathcal{H}_0^\gamma$ of $\mathcal{H}^\gamma$ by
$$\mathcal{H}_0^\gamma=\{u\in\mathcal{H}^\gamma:\frac{\partial u}{\partial\nu}\leq0 \,\,\,{\rm along}\,\,\,\gamma\}.$$
$\mathcal{H}_0^\gamma$ is nonempty because $k-cy$ is in $\mathcal{H}_0^\gamma$ if $c$ is a sufficiently big constant.  

Now that we have an element $u$ in $\mathcal{H}^\gamma_0$, we can slowly deform the values of $u$ along $\partial\Sigma$ by increasing them on $\partial\Sigma\setminus\gamma $ and fixing them on $\gamma$ at zero. In this way, we can obtain a 1-parameter family of boundary values $b_t$ along $\partial\Sigma$, which determine a 1-parameter family of solutions $u_t$ of $L(u)=0$ in $\Sigma$ with $u_t|_{\partial\Sigma}=b_t$. Here we are hoping that the values of $\partial u_t/\partial\nu$ will increase, and fortunately, one of $u_t$ will satisfy $\frac{\partial u_t}{\partial\nu}=0\,\,{\rm along}\,\,\gamma.$ Then $k-u_t$ will be the harmonic function satisfying (\ref{cc2}).

Following this relatively intuitive idea, let us give a rigorous proof. Given $u\in\mathcal{H}_0^\gamma$, define
$$m_0(u)={\rm min}\,\{{\partial u}/{\partial\nu}|_\gamma\}$$
and
$$M_0={\rm max}\,\,\{m_0(u):\,u\in\mathcal{H}_0^\gamma\}.$$
Obviously 
$$m_0(u)\leq0\,\,{\rm for\,\,any}\,\,u\in\mathcal{H}_0^\gamma,\,\,\,\,{\rm and}\,\,\,\,M_0\leq 0.$$
We will show that $M_0=0$ and that there exists  $u_{\gamma}\in\mathcal{H}_0^\gamma$ such that $m_0(u_{\gamma})=0$. Then $k-u_\gamma$ will be the desired harmonic function on $\Sigma$ satisfying the Cauchy conditions (\ref{cc2}).

First, let's suppose $\Sigma$ is simply connected. Let $C^\omega(\gamma)$ be the set of all analytic functions on the free boundary $\gamma$ and $C^0(\partial\Sigma\setminus\gamma)$ the set of all continuous functions on $\partial\Sigma\setminus\gamma$. 
Given a bounded continuous function $e$  on $\partial\Sigma$ vanishing on $\gamma$, let $u_{e}$ be the unique solution of $Lu=0$ on $\Sigma$ satisfying the Dirichlet condition $u_{e}|_{\partial\Sigma}=e$. Define the {\it Dirichlet-to-Neumann map}
\begin{center}
$\mathcal{N}:C^0(\partial\Sigma\setminus\gamma)\rightarrow C^\omega(\gamma)$\,\, by\,\, $\mathcal{N}(e|_{\partial\Sigma\setminus\gamma})=\frac{\partial u_{e}}{\partial\nu}|_\gamma$.
\end{center} 
Then $\mathcal{N}$ is linear, and the boundary point lemma(Hopf lemma) tells us that $\mathcal{N}$ is one-to-one and order-preserving, that is, if ${e}_1,{e}_2\in C^0(\partial\Sigma\setminus\gamma)$, ${e}_1\leq {e}_2$ on $\partial\Sigma\setminus\gamma$ and ${e}_1<{e}_2$ on a nonempty open subset of $\partial\Sigma\setminus\gamma$, then $\mathcal{N}(e_1)$ is strictly smaller than $\mathcal{N}(e_2)$ on $\gamma$.

Suppose $M_0<0$, and let's derive a contradiction. For any $\eta>0$ there exists $u_\eta\in \mathcal{H}_0^\gamma$ such that 
$$M_0-\eta\leq\frac{\partial u_{\eta}}{\partial\nu}\leq0\,\,\,\,{\rm along}\,\,\,\gamma.$$ 
If $M_1$ denotes the maximum of $\frac{\partial u_{\eta}}{\partial\nu}(x,0)$ for small $\eta$, then $M_0<M_1\leq0$. That is because if $M_1\leq M_0$, then there exists a harmonic function $\bar{h}(x,y)=cy(c>0)$ in $\Sigma$ such that $u_\eta+\bar{h}$ is in $\mathcal{H}^\gamma$ and satisfies
$$M_0-\eta+\frac{c}{F(x,0)}\leq\frac{\partial(u_\eta +\bar{h})}{\partial\nu}(x,0)\leq M_1+\frac{c}{F(x,0)}\leq  M_0+\frac{c}{F(x,0)}.$$  
So there exist a sufficiently small $\eta$ and some $c>0$ such that 
$$M_0<M_0-\eta+\frac{c}{{\rm max} \{F(x,0)\}},\,\,\,\,M_0<m_0(u_\eta +\bar{h}),$$
and
$$M_0+\frac{c}{{\rm min}\{F(x,0)\}}<0,\,\,\,\,u_\eta+\bar{h}\in\mathcal{H}_0^\gamma,$$
which contradicts that $M_0$ is the maximum among all $m_0(u),u\in\mathcal{H}_0^\gamma$. Therefore $M_1$ should be greater than $M_0$.

Define a smooth non-constant function ${f}(x)$ on $\gamma$ by
$$f(x)=\frac{2}{5}\left(\frac{M_0+M_1}{2}-\frac{\partial u_{\eta}}{\partial\nu}(x,0)\right)F(x,0),$$
and set $b=\frac{M_1-M_0}{5}>0$. Then
\begin{equation}\label{Mb}
M_0+b-\frac{3}{5}\eta\leq\frac{\partial u_{\eta}}{\partial\nu}(x,0)+\frac{f(x)}{F(x,0)}\leq M_1-b. 
\end{equation}
Let $b_0=\frac{f(1)-f(-1)}{2}$.
Then $f(x)-b_0x$ can be extended to a continuous periodic function with period 2 on the $x$-axis; so its Fourier series will be written as
$$f(x)-b_0x=\frac{a_0}{2}+\sum_{n=1}^\infty(a_n\cos n\pi x+b_n\sin n\pi x).$$
Define 
$$f_m(x)=\frac{a_0}{2}+\sum_{n=1}^m(a_n\cos n\pi x+b_n\sin n\pi x).$$
Then $\{f_m(x)\}$ converges absolutely to $f(x)$, and for any $\varepsilon>0$ there exists $k$ such that
\begin{equation}\label{f_k}
|f(x)-b_0x-f_k(x)|<\varepsilon.
\end{equation}
Since $f_k(x)$ is a finite Fourier sum, its Taylor series 
at $x=a$ has an infinite radius of convergence for any $a$:
\begin{equation}\label{Taylor}
    f_k(x)= f_k(a)+f_k'(a)(x-a)+ \frac{f_k''(a)}{2!}(x-a)^2+\frac{f_k'''(a)}{3!}(x-a)^3+\frac{f_k^{(4)}(a)}{4!}(x-a)^4+\cdots.
\end{equation} 

Now, we use the convergence of the Taylor series of $f_k(x)$ to find the desired harmonic function. Remember that the term-by-term integration of (\ref{Taylor}) converges absolutely. Hence the following series also converges for all $-\infty<x<\infty$:
\begin{equation*}
f_k(a)(x-a)-\frac{f_k''(a)}{3!}(x-a)^3+\frac{f_k^{(4)}(a)}{5!}(x-a)^5-\frac{f_k^{(6)}(a)}{7!}(x-a)^7+\cdots+(-1)^k\frac{f_k^{(2k)}(a)}{(2k+1)!}(x-a)^{2k+1}+\cdots.
\end{equation*}
For each $a$ we can define a function $h_a(y)$ on the vertical line $\{(a,y):-\infty<y<\infty\}$:
$$h_a(y)=(b_0a+f_k(a))y-\frac{f_k''(a)}{3!}y^3+\frac{f_k^{(4)}(a)}{5!}y^5-\frac{f_k^{(6)}(a)}{7!}y^7+\cdots+(-1)^k\frac{f_k^{(2k)}(a)}{(2k+1)!}y^{2k+1}+\cdots.$$
Therefore we have an entire function $h(x,y)$ with two variables on $\mathbb R^2$:
\begin{equation}\label{h(x,y)}
h(x,y)=(b_0x+f_k(x))y-\frac{f_k''(x)}{3!}y^3+\frac{f_k^{(4)}(x)}{5!}y^5-\frac{f_k^{(6)}(x)}{7!}y^7+\cdots+(-1)^k\frac{f_k^{(2k)}(x)}{(2k+1)!}y^{2k+1}+\cdots.
\end{equation}
Clearly, $h(x,y)$ is well-defined in the entire plane and $h(x,y)-b_0xy$ is periodic in $x$.
Moreover, $h(x,y)$ is harmonic because
\begin{eqnarray*}
h_{xx}&=&f_k''(x)y-\frac{f_k^{(4)}(x)}{3!}y^3+\frac{f_k^{(6)}(x)}{5!}y^5-\cdots+(-1)^k\frac{f_k^{(2k+2)}(x)}{(2k+1)!}y^{2k+1}+\cdots\\
&=&-h_{yy}.
\end{eqnarray*}
Thus the entire harmonic function $h(x,y)$ satisfies the Cauchy conditions on $\Sigma$:
$$h(x,0)=0\,\,\,{\rm and}\,\,\,\frac{\partial h}{\partial\nu}(x,0)=\frac{\partial h}{\partial y}(x,0)\,\frac{\partial y}{\partial\nu}(x,0)=\frac{b_0x+f_k(x)}{F(x,0)}.$$
Then by (\ref{f_k}) we have
$$\frac{b_0x+f_k(x)-\varepsilon}{F(x,0)}<\frac{f(x)}{F(x,0)}<\frac{b_0x+f_k(x)+\varepsilon}{F(x,0)}$$
and hence by (\ref{Mb}) we get
$$M_0+b-\frac{3}{5}\eta-\frac{\varepsilon}{F(x,0)}<\frac{\partial u_{\eta}}{\partial\nu}(x,0)+\frac{b_0x+f_k(x)}{F(x,0)}< M_1-b+\frac{\varepsilon}{F(x,0)}.$$
Therefore if $\eta$ and $\varepsilon$ are sufficiently small, we have
$$M_0<M_0+\frac{b}{2}<\frac{\partial(u_{\eta}+h)}{\partial\nu}(x,0)< M_1-\frac{b}{2}<0.$$
Thus
$$u_{\eta}+h\in\mathcal{H}^\gamma_0\,\,\,\,{\rm and}\,\,\,\,M_0+\frac{b}{2}<m_0(u_{\eta}+h),$$
which contradicts the assumption that $M_0$ is the maximum among 
$\{m_0(u):u\in\mathcal{H}_0^\gamma\}$.
Therefore $M_0=0$. 

Finally, it remains to show the existence of a function $u_{\gamma}\in\mathcal{H}_0^\gamma$ with $m_0(u_{\gamma})=0$.
Let $\{u_n(x,y)\}$ be a sequence of functions in $\mathcal{H}_0^\gamma$ such that $-\frac{1}{n}<m_0(u_n)<0$. Define a (piecewise) continuous function on $\partial\Sigma$:
$$b_{\partial\Sigma}=\limsup_{n\rightarrow\infty}\,u_n|_{\partial\Sigma},$$
and let $u_\gamma(x,y)$ be a unique function in $\Sigma$ with $u_\gamma|_{\partial\Sigma}=b_{\partial\Sigma}$. It may happen that $u_\gamma\notin\mathcal{H}_0^\gamma$ because $b_{\partial\Sigma}$ can be infinite at some points of $\partial\Sigma\setminus\gamma$, but $u_\gamma$ is finite in $\Sigma\setminus\partial\Sigma$ and a solution of $Lu=0$. Since $u_\gamma|_{\gamma}=u_n|_{\gamma}$ and $m_0(u_n)\rightarrow0$, we have
$$u_\gamma=0\,\,\,{\rm and}\,\,\,\frac{\partial u_\gamma}{\partial\nu}=0\,\,\,{\rm along}\,\,\,\gamma.$$
Therefore
$$\Delta(k-u_\gamma)=0\,\,\,{\rm in}\,\,\,\Sigma\,\,\,{\rm and}\,\,\,k-u_\gamma=0,\,\,\,\frac{\partial(k-u_\gamma)}{\partial\nu}=1\,\,\,{\rm along}\,\,\,\gamma.$$
Setting $h=k-u_\gamma$ completes the proof when $\Sigma$ is simply connected.

The same proof works also for doubly connected $\Sigma$. We do not need to subtract $b_0x$ from $f(x)$ because $f(x)$ can be directly lifted to a continuous periodic function on the $x$-axis.
\end{proof}

\begin{remark}
a) In the proof of Lemma \ref{har1}, we also found an entire solution $h(x,y)$ to the Cauchy problem $\Delta h(x,y)=0$, $h(x,0)=0$, $\frac{\partial h}{\partial y}(x,0)=f(x)$ in $\mathbb R^2$ by using the Taylor series of $f(x)$ in case its radius of convergence is infinite. But one can also use the Fouruer series: If $a_n,\,b_n$ are the Fourier coefficients of ${f(x)}$ with period $2$, then
$$h(x,y)=\frac{a_0}{2}y+\sum_{n=1}^\infty\frac{1}{n}\sinh ny({a_n}\cos n\pi x+{b_n}\sin n\pi x).$$

b) One can similarly find an entire harmonic solution $h(x,y)$ to the Cauchy problem $h(x,0)=g(x)$, $\frac{\partial h}{\partial y}(x,0)=0$ in $\mathbb R^2$ in two ways, provided $g(x)$ is an analytic periodic function with period $2$ that has a Taylor series with an infinite radius of convergence.
Given the Taylor series of $g(x)$ centered at $a$,
$$g(x)= g(a)+g'(a)(x-a)+ \frac{g''(a)}{2!}(x-a)^2+\frac{g'''(a)}{3!}(x-a)^3+\frac{g^{(4)}(a)}{4!}(x-a)^4+\cdots,$$
we extract
\begin{equation*}
g(a)-\frac{g''(a)}{2!}(x-a)^2+\frac{g^{(4)}(a)}{4!}(x-a)^4-\frac{g^{(6)}(a)}{6!}(x-a)^6+\cdots+(-1)^k\frac{g^{(2k)}(a)}{(2k)!}(x-a)^{2k}+\cdots.
\end{equation*}
Then, define a function $h_a(y)$ on the vertical line $\{(a,y):-\infty<y<\infty\}$ by
$$h_a(y)=g(a)-\frac{g''(a)}{2!}y^2+\frac{g^{(4)}(a)}{4!}y^4-\frac{g^{(6)}(a)}{6!}y^6+\cdots+(-1)^k\frac{g^{(2k)}(a)}{(2k)!}y^{2k}+\cdots,$$
and the solution is 
\begin{equation*}
h(x,y)=g(x)-\frac{g''(x)}{2!}y^2+\frac{g^{(4)}(x)}{4!}y^4-\frac{g^{(6)}(x)}{6!}y^6+\cdots+(-1)^k\frac{g^{(2k)}(x)}{(2k)!}y^{2k}+\cdots.
\end{equation*}
On the other hand, given $a_k$ and $b_k$ for the Fourier coefficients of $g(x)$,  the entire harmonic function $h(x,y)$  can be also written as
\begin{equation}\label{cosh}
h(x,y)=\frac{a_0}{2}+\sum_{n=1}^\infty\cosh ny(a_n\cos n\pi x+b_n\sin n\pi x).
\end{equation}
\end{remark}

Two special solutions to the Cauchy problems in $\mathbb R^2$ can generate a general solution as follows.

\begin{proposition}
Let $f(x),g(x)$ be analytic periodic functions whose Taylor series have an infinite radius of convergence on the $x$-axis of $\mathbb R^2$. Then the solution $h(x,y)$ to the Cauchy problem in $\mathbb R^2$
$$\Delta h(x,y)=0,\,\,\,\,h(x,0)=g(x),\,\,\,\,\frac{\partial h}{\partial y}(x,0)=f(x)$$ can be written as
$$h(x,y)=\sum_{k=0}^\infty \left\{ (-1)^k\frac{g^{(2k)}(x)}{(2k)!}y^{2k}+(-1)^k\frac{f_k^{(2k)}(x)}{(2k+1)!}y^{2k+1}\right\}.$$
\end{proposition}

Going back to $\Sigma$, Lemma \ref{har1} gives us desired isothermal coordinates:

\begin{lemma}\label{ic}
Let $\Sigma$ be a simply connected or doubly connected minimal surface with free boundary $\gamma$ in a unit ball of $\mathbb R^3$. Then there exist isothermal coordinates $X,Y$ in $\Sigma$ away from the punctures $p_1,\ldots,p_n$ of $\Sigma$ such that $$ds^2=F(X,Y)^2(dX^2+dY^2)\,\,\,\,{\rm in}\,\,\,\,\Sigma, $$
and
$$Y=0,\,\,\,\,F(X,0)=1 \,\,\,\,{\rm along}\,\,\,\,\gamma.$$
\end{lemma}

\begin{proof}
From Lemma \ref{har1} we get a harmonic function $h$ in $\Sigma$ with
\begin{equation}\label{F=1}
h=0,\,\,\,\frac{\partial h}{\partial\nu}=1\,\,\,{\rm along}\,\,\,\gamma.
\end{equation}
Let $h^*$ be the harmonic function that is conjugate to $h$. Then $\{X,Y\}:=\{h^*,h\}$ become isothermal coordinates on $\Sigma$ such that $Y=0$ along $\gamma$. If we write the metric of $\Sigma$ as
$$ds^2=F(X,Y)^2(dX^2+dY^2),$$
then (\ref{F=1}) implies
$$F(X,0)=1\,\,\,{\rm along}\,\,\,\gamma.$$
In the domain $D$ (as in Lemma \ref{har1}), the function $Z=H(z)$ mapping $x+iy$ to $X+iY$ is holomorphic. $z$ is called a branch point of $H(z)$ if $H'(z)=0 $. Then $(H^{-1})'(Z)=\infty$ if $H^{-1}(Z)$ is a branch point of $H(z)$. Hence 
$$0< F(X,Y)\leq\infty.$$
But $X,Y$ can be called isothermal coordinates only if 
$0< F(X,Y)<\infty.$
So the branch points of $H(z)$ will be deleted from $\Sigma$ and called the {\it punctures} of $\Sigma$.

When $\Sigma$ is doubly connected, $X$ is multi-valued in $\Sigma$ and periodic in $D=\mathbb R\times(0,a)$, and hence  $F(X,Y)^2(dX^2+dY^2)$ is well defined in $\Sigma$.
\end{proof}

\section{reflection across a sphere}
With the specific isothermal parameters, we can now introduce the reflection principle for free boundary minimal surfaces in a ball.

\begin{theorem}\label{reflection}
Let $\Sigma$ be an immersed, simply connected or doubly connected minimal surface in the unit ball $B\subset\mathbb R^3$. In case $\Sigma$ is simply connected, assume that $\gamma:=\partial\Sigma\cap\partial B$ is connected and immersed, and $\Sigma\cup\gamma$ is $C^1$ such that $\Sigma$ is perpendicular to $\partial B$ along $\gamma$, that is, $\gamma$ is a free boundary of $\Sigma$ in $B$. If $\Sigma$ is doubly connected, $\gamma$ is assumed to be one of the two boundary components of $\Sigma$.

{\rm a)} Then there exists a minimal surface $\Sigma^{double}\supset\Sigma$ that is an analytic continuation of $\Sigma$ across $\gamma$ such that $\Sigma^*:=\Sigma^{double}\setminus\overline{\Sigma}$ is conformally equivalent to $\Sigma$.

{\rm b)} We call $\Sigma^*$ the \underline{spherical mirror image} of $\Sigma$ across $\partial B$.

{\rm c)} $\Sigma^*$ may have ends and in this case $\Sigma^*$ is conformally equivalent to $\Sigma$ with punctures.

{\rm d)} Similarly, a minimal surface $\Sigma$ outside $B$ with a free boundary in $\partial B$ can be reflected across $\partial B$.
\end{theorem}

\begin{proof}
As in Section 2, there is a conformal map $\Psi$ from a half unit disk $D:=\{(x,y):x^2+y^2<1,\,0<y\}$ onto the simply connected $\Sigma$, mapping $\delta$, the diameter of $D$, onto $\gamma$. $\Psi$ is harmonic as well since $\Sigma$ is minimal. Assume that $Z=H(z)$ is a holomorphic function on $D$ that gives the isothermal parameters $X,Y$  with $Z=X+iY$ and $z=x+iy$ such that
$$ds^2=F_H(X,Y)^2(dX^2+dY^2)$$
and
\begin{equation}\label{onee}
F_H(X,Y)=1\,\,\,{\rm on}\,\,\, H(\delta)\subset\{Y=0\}.
\end{equation}
Since $\Sigma$ is perpendicular to $\partial B$ along $\gamma$, we have
\begin{equation}\label{steklov}\Psi:=(\psi_1,\psi_2,\psi_3)=\left(\frac{\partial\psi_1}{\partial\nu},
\frac{\partial\psi_2}{\partial\nu},\frac{\partial\psi_3}{\partial\nu}\right)\,\,\,{\rm along}\,\,\,\gamma,
\end{equation}
where $\nu$ is the outward unit conormal to $\partial\Sigma$ on $\Sigma$.
From (\ref{onee}) we see that
$$\nu=-\frac{\partial}{\partial Y}\,\,\,{\rm along}\,\,\,\gamma.$$
({\it This is the reason why we  first proved} Lemma \ref{ic}.)
Introducing a reparametrization
$$\Phi:=(\phi_1,\phi_2,\phi_3)=\Psi\circ H^{-1}:H(D)\rightarrow\Sigma,$$
 we see that
\begin{equation}\label{ddv}
(\phi_1,\phi_2,\phi_3)=-\left(\frac{\partial\phi_1}{\partial Y},\frac{\partial\phi_2}{\partial Y},\frac{\partial\phi_3}{\partial Y}\right)\,\,\,{\rm along}\,\,\,\gamma.
\end{equation}
Since $\phi_j$ is harmonic in $X$ and $Y$, we can find a holomorphic function $\Phi_j(Z)$ on $H(D)$ with ${\rm Re}\Phi_j=\phi_j,\, j=1,2,3$. 
Then
$$-\frac{\partial\phi_j}{\partial Y}={\rm Im}\frac{\partial\Phi_j}{\partial Z}.$$
Hence  on the $X$-axis \eqref{ddv} implies,
\begin{equation}\label{schwarz}
{\rm Im}\left(i\Phi_j-\frac{\partial\Phi_j}{\partial Z}\right)=0.
\end{equation}
Note that the left hand side ${\rm Im}(i\Phi_j-{\partial\Phi_j}/{\partial Z})$ is a harmonic function on $H(D)$ vanishing on $\partial( H(D))\subset X$-axis . So far, we have transformed the Steklov condition (\ref{steklov}) into the Schwarz condition (\ref{schwarz}). By the Schwarz reflection principle the holomorphic function $i\Phi_j-\partial\Phi_j/\partial Z$ on $H(D)$ has a holomorphic extension $$\Lambda_j:=\lambda_j+i\lambda_j^*\,\,\,\,\,{\rm  over}\,\,\,\,\, H(D)\cup H(\delta)\cup H(D)^*,$$ where $\lambda_j,\lambda_j^*$ are harmonic conjugates and $H(D)^*$ is the mirror image of $H(D)$ across the $X$-axis.

Does $ \Phi_j$ also have a holomorphic extension over $H(D)\cup H(\delta)\cup H(D)^*$? Yes, it surely does! One can obtain the holomorphic extension of $\Phi_j$ by solving the first-order linear differential equation
 $$i\Phi_j-\frac{\partial\Phi_j}{\partial Z}=\Lambda_j.$$
Clearly,
$$\frac{\partial}{\partial Z}(e^{-iZ}\Phi_j)=-e^{-iZ}\Lambda_j,$$
hence
$$\Phi_j=-e^{iZ}\int e^{-iZ}\Lambda_jdZ.$$
$\Phi_j$ involves two arbitrary constants: one arising from $\lambda_j$ and the other from the integration of $e^{iZ}\Lambda_j$. We choose the correct constants which give us ${\rm Re}\,\Phi_j|_{H(D)}=\phi_j$.
Therefore
$$\phi_j:={\rm Re}\left(-e^{iZ}\int e^{-iZ}\Lambda_jdZ\right)$$
is the desired harmonic extension of $\phi_j$ over $H(D)\cup H(\delta)\cup H(D)^*$.

From the symmetry of the holomorphic function $H$, we know that $H$ maps $D\cup\delta\cup D^*$ onto $H(D)\cup H(\delta)\cup H(D)^*$. Hence
$$\Psi^1:={\rm Re}(\Phi_1\circ H,\Phi_2\circ H, \Phi_3\circ H)$$
is a harmonic map defined on $D\cup\delta\cup D^*$ and is an analytic continuation of the original conformal harmonic map $\Psi:D\rightarrow\Sigma$. Composing with $H$ gives an analytic continuation in the original isothermal coordinates $x,y$.

Since $\Psi=(\psi_1,\psi_2,\psi_3):D\rightarrow\Sigma$ is conformal, we have
$$\left(\frac{d\psi_1}{dz}\right)^2+\left(\frac{d\psi_2}{dz}\right)^2+\left(\frac{d\psi_3}{dz}\right)^2=0\,\,\,\,{\rm in}\,\,\,\,D.$$
As $\Psi^1:=(\Psi_1^1,\Psi_2^1,\Psi_3^1)$ is the harmonic extension of $\Psi$ in $D\cup\delta\cup D^*$,
we also have
$$\left(\frac{d\Psi^1_1}{dz}\right)^2+\left(\frac{d\Psi^1_2}{dz}\right)^2+\left(\frac{d\Psi^1_3}{dz}\right)^2=0\,\,\,\,{\rm in}\,\,\,\,D\cup\delta\cup D^*.$$
Hence $\Psi^1$ is conformal as well. So $\Psi^1(D\cup\delta\cup D^*)$ is minimal.
Define
$$\Sigma^{double}=\Psi^1(D\cup\delta\cup D^*)\,\,\,\,{\rm and}\,\,\,\,\Sigma^*=\Psi^1(D^*).$$
Then $\Sigma^*$ is the desired spherical mirror image of $\Sigma$ that is conformally equivalent to $\Sigma$.

Suppose now that $\Sigma$ is a minimal annulus. There exists a periodic conformal harmonic map $\Psi$ of $\mathbb R\times(0,a)$ onto $\Sigma$ for some $a>0$. Then, as above, we can find $\Psi^1:\mathbb R\times(-a,a)\rightarrow\mathbb R^3$, an analytic continuation of $\Psi$. Here, we must show that $\Psi^1$ is periodic in $x$. Consider the vector-valued harmonic function on $\mathbb R\times(-a,a)$
$$\Delta(x,y):=\Psi^1(x+2\pi,y)-\Psi^1(x,y).$$
Then at $(x,y)\in \mathbb R\times(0,a)$ we have
\begin{eqnarray*}
\Delta(x,y)&=&\Psi(x+2\pi,y)-\Psi(x,y)\,\,\,\,\,\,\,[{\rm because}\,\,\Psi^1\equiv\Psi\,\,{\rm on}\,\, \mathbb R\times(0,a)]\\
&=&0\,\,\,\,\,\,\,[{\rm because}\,\, \Psi\,\,{\rm is\,\,periodic\,\,in}\,\,x\,\,{\rm with\,\,period\,\,}2\pi\,\,{\rm on}\,\, \mathbb R\times(0,a)].
\end{eqnarray*}
Hence $\Delta\equiv0$ on $\mathbb R\times(0,a)$ and since $\Delta$ is harmonic, $\Delta$ vanishes on $\mathbb R\times(-a,a)$ as well, meaning that $\Psi^1$ is also periodic in $x$ on $\mathbb R\times(-a,a)$ with period $2\pi$.

Define $\Sigma^{double}:=\Psi^1(\mathbb R\times(-a,a))$ and $\Sigma^*:=\Psi^1(\mathbb R\times(-a,0))$. Then $\Sigma^*$ is the desired { spherical mirror image} of $\Sigma$, and clearly, $\Sigma^*$ is conformally equivalent to $\Sigma$, possibly with punctures. Some punctures of $\Sigma^*$ will correspond to the ends because the conformal factor $F(X,Y)$ can become infinite at those punctures.

Whether $\Sigma$ is inside $B$ or outside $B$, $\Sigma$ can be reflected across $\partial B$ as long as its free boundary $\gamma$ lies inside $\partial B$. This is because (\ref{steklov}) and (\ref{ddv}) still hold (with the opposite sign) whether $\Sigma\subset B$ or $\Sigma\subset B^c$.
\end{proof}

\section{repeated reflections}
There is a substantial difference between the case where $\Sigma$ has only one free boundary $\gamma$ and the case where $\Sigma$ has two free boundary components $\gamma_1$ and $\gamma_2$. In the first case, $\Sigma$ can be extended only once across $\gamma$, but in the second case, it can be extended infinitely many times, across $\gamma_1$ and $\gamma_2$ alternatingly.

Let $\Sigma$ be a  minimal annulus with a free boundary $\gamma$ in a ball $B\subset\mathbb R^3$ and assume that  $\gamma$ is a closed loop.  There exists a periodic conformal harmonic map $\Psi$ from $V:=\mathbb R\times(0,a)$ onto $\Sigma$ such that $\Psi$ has period $2\pi$ in $x$ and maps the $x$-axis onto $\gamma$. As in Theorem \ref{reflection}, $\Psi$ is extended by the Schwarz reflection principle to $\Psi^1$ on $\mathbb R\times(-a,a)$ so that $\Psi^1(\mathbb R\times(-a,a))$ is the analytic continuation of $\Sigma$ containing the spherical mirror image of $\Sigma$, $\Sigma^*=\Psi^1(\mathbb R\times(-a,0))$. Let's call $\Psi^1(\mathbb R\times(-a,a))$ the {\it double extension } of $\Sigma=\Psi(\mathbb R\times (0,a))$ across $\gamma=\Psi(\mathbb R\times\{0\})$. Denote the double extension of $\Sigma$ across $\gamma$ by $\Sigma^\gamma$. One can say that $\gamma$ is conformally in the center of $\Sigma^\gamma$. Let's call $\gamma$ the {\it line of reflection}. $\Psi^1$ may have punctures in $\mathbb R\times(-a,a)$. In that case, the set of punctures of $\Psi^1$ will be denoted as ${\mathcal{P}}_1$ and $\Sigma^\gamma$ will be the image $\Psi^1(\mathbb R\times(-a,a)\setminus{\mathcal{P}}_1)$.

Suppose $\Sigma$ is a minimal annulus with two boundary components $\gamma_1$ and $\gamma_2$, free on $\partial B$. Suppose also that $\Psi$ is a periodic conformal harmonic map  from $V:=\mathbb R\times(0,a)$ onto $\Sigma$ with period $2\pi$ and  $\gamma_1=\Psi(\mathbb R\times\{0\})$. As above, $\Psi$ defined on $\mathbb R\times(0,a)$ extends to $\Psi^1$ on $\mathbb R\times(-a,a)\setminus{\mathcal{P}}_1$ by the reflection across $\gamma_1$, and $\Sigma$ has a double extension $\Sigma^{\gamma_1}=\Psi^1(\mathbb R\times(-a,a)\setminus{\mathcal{P}}_1)$. The starting point of the infinite reflections is that $\gamma_1$ is in the center of $\Sigma^{\gamma_1}$ whereas $\gamma_2$ is not. By alternating $\gamma_1$ and $\gamma_2$ as the line of reflection, one can extend $\Sigma$ infinitely many times.

Since $\gamma_2$ is a free boundary of $\Sigma^{\gamma_1}$, by Theorem \ref{reflection}  we can reflect $\Sigma^{\gamma_1}$ across $\gamma_2:=\Psi^1(\mathbb R\times \{a\})$ to get a periodic analytic continuation $\Psi^2$ of $\Psi^1$ defined on $\mathbb R\times(-a,3a)$. ($\Psi^2$ does not mean $\Psi\circ\Psi$.) Then $\Sigma^{\gamma_1\gamma_2}:=\Psi^2(\mathbb R\times(-a,3a))$ is the double extension of $\Sigma^{\gamma_1}$. More precisely, considering the set of punctures of $\Psi^2$ denoted ${\mathcal{P}}_2\subset\mathbb R\times(-a,3a)$, $\Sigma^{\gamma_1\gamma_2}$ will be $\Psi^2(\mathbb R\times(-a,3a)\setminus{\mathcal{P}}_2)$. This time $\gamma_2$ is in the center of $\Sigma^{\gamma_1\gamma_2}$ whereas $\gamma_1$ is not. Note here that $\gamma_1=\Psi(\mathbb R\times\{0\})$ is the free boundary of the subset $\Psi^2(\mathbb R\times(0,3a))$ of $\Sigma^{\gamma_1\gamma_2}$. So let's apply Theorem \ref{reflection}  to $\Psi^2(\mathbb R\times(0,3a))$ to extend $\Psi^2$ to $\Psi^3$ periodically on the strip $\mathbb R\times(-3a,3a)\setminus{\mathcal{P}}_3$. Hence $\Sigma^{\gamma_1^2\gamma_2}:=\Psi^3(\mathbb R\times(-3a,3a)\setminus{\mathcal{P}}_3)$ is the double extension of $\Psi^2(\mathbb R\times(0,3a))$ and is an analytic continuation of $\Sigma^{\gamma_1\gamma_2}$. Now $\gamma_2$ is not in the center of $\Sigma^{\gamma_1^2\gamma_2}$ and is the free boundary of the subset $\Psi^3(\mathbb R\times(-3a,a)$, so again we apply Theorem \ref{reflection}  to get a periodic conformal harmonic map $\Psi^4$ defined on $\mathbb R\times(-3a,5a)\setminus{\mathcal{P}}_4$. $\Sigma^{\gamma_1^2\gamma_2^2}:=\Psi^4(\mathbb R\times(-3a,5a)\setminus{\mathcal{P}}_4)$ is an analytic continuation of $\Sigma^{\gamma_1^2\gamma_2}$.

From here, let's proceed by induction. Suppose there is a periodic conformal harmonic map $\Psi^{2k}$ defined on $\mathbb R\times(-(2k-1)a,(2k+1)a)\setminus{\mathcal{P}}_{2k}$ with period $2\pi$ and extending the original $\Psi$ on $\mathbb R\times(0,a)$. Denote $\Sigma^{\gamma_1^{k}\gamma_2^{k}}:=\Psi^{2k}(\mathbb R\times(-(2k-1)a,(2k+1)a)\setminus{\mathcal{P}}_{2k})$. $\gamma_1$ is not in the center of $\Sigma^{\gamma_1^{k}\gamma_2^{k}}$ and is the free boundary of $\Psi^{2k}(\mathbb R\times(0,(2k+1)a))$. So using Theorem \ref{reflection} , we can reflect $\Psi^{2k}(\mathbb R\times(0,(2k+1)a))$ across $\gamma_1$ and extend $\Psi^{2k}$  to $\Psi^{2k+1}$ periodically on $\mathbb R\times(-(2k+1)a,(2k+1)a)\setminus{\mathcal{P}}_{2k+1}$ with period $2\pi$. Then $\Sigma^{\gamma_1^{k+1}\gamma_2^{k}}:=\Psi^{2k+1}(\mathbb R\times(-(2k+1)a,(2k+1)a)\setminus{\mathcal{P}}_{2k+1})$ is an analytic continuation of $\Sigma^{\gamma_1^{k}\gamma_2^{k}}$. Again reflect $\Psi^{2k+1}(\mathbb R\times(-(2k+1)a,a))$ across its free boundary $\gamma_2$ by Theorem \ref{reflection}  to get a periodic analytic continuation $\Psi^{2k+2}$ of $\Psi^{2k+1}$ to $\mathbb R\times(-(2k+1)a,(2k+3)a)\setminus{\mathcal{P}}_{2k+2}$ with period $2\pi$ and with $\Sigma^{\gamma_1^{k+1}\gamma_2^{k+1}}:=\Psi^{2k+2}(\mathbb R\times(-(2k+1)a,(2k+3)a)\setminus{\mathcal{P}}_{2k+2})$. Obviously  $\Sigma^{\gamma_1^{k+1}\gamma_2^k}\subset\Sigma^{\gamma_1^{k+1}\gamma_2^{k+1}}$.

Define a conformal harmonic map $\widetilde{\Psi}^n$ on the annulus $\{w\in\mathbb C:e^{-(n-1)a}<|w|<e^{(n-1)a}\}$ by
$$\widetilde{\Psi}^{n}(w):=\Psi^n(i\log w).$$
Even though $\log w$ is many-valued, $\widetilde{\Psi}^n(w)$ is well-defined because $\Psi^n(x+iy)$ is periodic in $x$ with period $2\pi$.
Hence we have an increasing sequence of minimal surfaces $\{\widetilde{\Sigma}^n\}$ defined by
$$\widetilde{\Sigma}^n:=\widetilde{\Psi}^n\left(\{w:e^{-(n-1)a}<|w|<e^{(n-1)a}\}\right),\,\,\,
\Sigma\subset\widetilde{\Sigma}^2\subset\widetilde{\Sigma}^3\subset\cdots\subset\widetilde{\Sigma}^n\subset\cdots.$$
It follows that $$\widetilde{\Sigma}^n=\Psi^n(\mathbb R\times(-(n-1)a,(n-1)a)\setminus{\mathcal{P}}_n)\,\,\,\,\, {\rm and} \,\,\,\,\,
\widetilde{\Sigma}^{2k}\subset\Sigma^{\gamma_1^k\gamma_2^k}\subset\widetilde{\Sigma}^{2k+2}.$$
So the limiting surface
$$\widetilde{\Sigma}:=\lim_{n\rightarrow\infty}\widetilde{\Sigma}^n$$
exists and equals $\lim_{k\rightarrow\infty}\Sigma^{\gamma_1^k\gamma_2^k}$. Furthermore $\widetilde{\Sigma}$ is conformally equivalent to $\mathbb R^2,\,\mathbb S^1\times\mathbb R$, or $\mathbb S^2$ with punctures. The corresponding conformal harmonic map should also exist:
 $$\widetilde{\Psi}:=\lim_{n\rightarrow\infty}\widetilde{\Psi}^n.$$
One can see that
$$\widetilde{\Sigma}=\widetilde{\Psi}\left(\mathbb R^2\setminus(\{O\}\cup{\widetilde{\mathcal{P}}})\right),\,\,\,\widetilde{\mathcal{P}}:=\bigcup_n\bigcup_{z\in\mathcal{P}_n}e^{-iz}
.$$

\section{critical catenoid}
So far, the free boundary property has been the key to establishing the reflection principle. Moreover, the complex function theory has been the primary tool in constructing the minimal surface $\widetilde{\Sigma}$. Henceforth, we will show that $\widetilde{\Sigma}$ is the catenoid if the original $\Sigma$ is embedded. Then $\Sigma$ will have to be the critical catenoid. From here, differential geometry will be the primary tool.

In this section we will follow the arguments of \cite{DHKW} and use their notations. Let $X$ be a smooth surface (as a map) with isothermal coordinates  $u,v$ such that $w=u+iv$ is a complex coordinate on $X$. Denote the metric of $X$ by $ds^2=\Lambda(u,v)(du^2+dv^2)$. Let $\vec{n}$ be a unit normal to $X$ in $\mathbb R^3$ and let $\mathcal{L},\mathcal{M},\mathcal{N}$ be the components of the second fundamental form of  $X$, $H$ the mean curvature, and $K$ the Gaussian curvature of $X$. We have the following from the Lemma of Section 1.3 in \cite{DHKW}.

\begin{lemma}\label{DHKW}
$X$ satisfies
$$X_{uu}=\frac{\Lambda_u}{2\Lambda}X_u-\frac{\Lambda_v}{2\Lambda}X_v+\mathcal{L}\,\vec{n}\,\,\,$$
\begin{equation}\label{rectangle}X_{uv}=\frac{\Lambda_v}{2\Lambda}X_u+\frac{\Lambda_u}{2\Lambda}X_v+\mathcal{M}\,\vec{n}
\end{equation}
$$\,\,\,X_{vv}=-\frac{\Lambda_u}{2\Lambda}X_u+\frac{\Lambda_v}{2\Lambda}X_v+\mathcal{N}\,\vec{n}$$
$$\,\,\,\,\,\,\,H=\frac{\mathcal{L}+\mathcal{N}}{2\Lambda},\,\,\,\,\,\,K=\frac{\mathcal{L}\mathcal{N}-\mathcal{M}^2}{\Lambda^2}$$
\begin{equation}\label{holo}
\left[\frac{1}{2}(\mathcal{L}-\mathcal{N})-i\mathcal{M}\right]_{\bar{w}}=\Lambda H_w.
\end{equation}
Define $$f(w):=\frac{1}{2}(\mathcal{L}-\mathcal{N})-i\mathcal{M},\,\,\,\,\,\,\alpha:={\rm Re}\,[w^2f(w)],\,\,\,\,\,\,\beta:={\rm Im}\,[w^2f(w)].$$
{\rm (\ref{holo})} implies that $f(w)$ is holomorphic if the surface $X$ has constant mean curvature.

Let $\rho,\theta$ be the polar coordinates on $\mathbb R^2$ such that $w=u+iv=\rho e^{i\theta}$. Then {\rm (\ref{rectangle})} can be rewritten as
$$\,\,\,\,\,\,\,\,X_{\rho\rho}=\frac{\Lambda_\rho}{2\Lambda}X_\rho-\frac{1}{\rho}\frac{\Lambda_\theta}{2\Lambda}\frac{1}{\rho}X_\theta+\left( \frac{\alpha}{\rho^2}+\Lambda H\right)\vec{n}$$
\begin{equation}\label{polar}\frac{1}{\rho}X_{\rho\theta}=\frac{1}{\rho}\frac{\Lambda_\theta}{2\Lambda}X_\rho+\left(\frac{1}{\rho}+
\frac{\Lambda_\rho}{2\Lambda}\right)\frac{1}{\rho}X_\theta-\frac{\beta}{\rho^2}\,\vec{n}
\end{equation}
 $$\hspace{2.19cm}\frac{1}{\rho^2}X_{\theta\theta}=-\left(\frac{1}{\rho}+
\frac{\Lambda_\rho}{2\Lambda}\right)X_\rho+
\frac{1}{\rho}\frac{\Lambda_\theta}{2\Lambda}\frac{1}{\rho}X_\theta-\left(\frac{\alpha}{\rho^2}-\Lambda H\right)\vec{n}.$$
\end{lemma}

\begin{lemma}\label{K}
Let $\Sigma$ be an immersed minimal annulus in a unit ball $B\subset\mathbb R^3$ with free boundary $\partial\Sigma\subset\partial B$. If $\widetilde{\Sigma}$ is the analytic continuation of $\Sigma$ obtained after infinite reflections as in the preceding section, then the Gaussian curvature $K$ is nowhere zero on $\widetilde{\Sigma}$.
\end{lemma}
\begin{proof}
(\ref{polar}) will give important information on the free boundary $\gamma_1\cup\gamma_2$ of $\Sigma$. Recall that $\rho=1$ on $\gamma_1$ and $\rho=e^a$ on $\gamma_2$. First, we know that
$$X_\rho=\sqrt{\Lambda}\,X\,\,\,{\rm on}\,\,\,\gamma_1\cup\gamma_2.$$
Next, we differentiate this equation with respect to $\theta$:
\begin{equation}\label{rt}
X_{\rho\theta}=(\sqrt{\Lambda})_\theta X+\sqrt{\Lambda}X_\theta=
\frac{(\sqrt{\Lambda})_{\theta}}{\sqrt{\Lambda}}X_\rho+\sqrt{\Lambda}X_\theta\,\,\,{\rm on}\,\,\,\gamma_1\cup\gamma_2.
\end{equation}
Now let's apply Lemma \ref{DHKW} when $X=\widetilde{\Psi}$ and $\Lambda=F^2$. Then $w^2f(w)=\alpha+i\beta$ is a holomorphic function on $\widetilde{\Sigma}$. Compare (\ref{rt}) with (\ref{polar}) to get
$$\frac{1}{\rho}+\frac{F_\rho}{F}={F}\,\,\,{\rm and}\,\,\,\beta=0\,\,\,{\rm on}\,\,\,\gamma_1\cup\gamma_2.$$
Hence $\beta\equiv0$ on $\Sigma$. Remember that $\beta$, being harmonic, extends to $\widetilde{\Sigma}$. Therefore $\beta\equiv0$ on $\widetilde{\Sigma}$ as well, and hence $\alpha$ is a constant $c$ on $\widetilde{\Sigma}$. $c$ is nonzero because otherwise $\Sigma$ would be flat. (See Theorem, p.343, \cite{DHKW}.)
By Lemma \ref{DHKW}
$$K=-\frac{\mathcal{L}^2+\mathcal{M}^2}{F^4}=-\frac{|f|^2}{F^4}=-\left|\frac{\alpha+i\beta}{w^2}\right|^2 \frac{1}{F^4}=-\left|\frac{c}{w^2}\right|^2\frac{1}{F^4}.$$
Therefore $K<0$ everywhere on $\widetilde{\Sigma}$ because $0<|w|<\infty$ on $\widetilde{\Sigma}$.
\end{proof}

A point of a minimal surface is called a flat point if the Gaussian curvature $K$ vanishes at that point. The flat points are isolated on a minimal surface. If a minimal surface is in $\mathbb R^3$, then the flat point is a  point at which the derivative of the  Gauss map vanishes. Hence Lemma \ref{K} implies that the Gauss map is a local diffeomorphism everywhere on $\widetilde{\Sigma}$. Let $\widetilde{\Psi}$ be the immersion from $\mathbb R^2\setminus(\{O\}\cup{\widetilde{\mathcal{P}}})$ onto $\widetilde{\Sigma}$ and $G$ the Gauss map from $\widetilde{\Sigma}$ to $\mathbb S^2$. Then $G\circ\widetilde{\Psi}$ is a covering map from $\mathbb R^2\setminus(\{O\}\cup{\widetilde{\mathcal{P}}})$ onto its image $G\circ\widetilde{\Psi}(\mathbb R^2\setminus(\{O\}\cup{\widetilde{\mathcal{P}}}))\subset\mathbb S^2$. $G\circ\widetilde{\Psi}$ is also a conformal map(with the opposite orientation of $\mathbb S^2$).
\begin{definition}
Denote by $n,s$ the north and south poles of $\mathbb S^2$, respectively. Let $\pi_s$ be the stereographic projection from $\mathbb S^2\setminus\{n,s\}$ onto $\mathbb R^2\setminus\{O\}$, mapping a neighborhood of $s$ to a neighborhood of $O$. Then $G\circ\widetilde{\Psi}\circ\pi_s$ is a conformal covering map from $\mathbb S^2\setminus(\{n,s\}\cup\pi_s^{-1}(\widetilde{\mathcal{P}}))$ into $\mathbb S^2$. \\
\end{definition}
The free boundary $\partial\Sigma$ is the union of lines of curvatures $\gamma_1,\gamma_2$ on $\partial B$. Since $K$ is negative on $\partial\Sigma$, the principal curvatures are nonzero along $\partial\Sigma$. Hence both $\gamma_1$ and $\gamma_2$ are locally strictly convex(in the opposite directions) on $\partial B$. Since  $G\circ\widetilde{\Psi}\circ\pi_s$ is a conformal covering map from $\mathbb S^2\setminus(\{n,s\}\cup\pi_s^{-1}(\widetilde{\mathcal{P}}))$ into $\mathbb S^2$, $ G\circ\widetilde{\Psi}\circ\pi_s$ extends across the punctures $\{n,s\}\cup\pi_s^{-1}(\widetilde{\mathcal{P}})$ to a holomorphic map from $\mathbb S^2$ to $\mathbb S^2$.

\begin{lemma}
$\widetilde{\Sigma}$ is complete.
\end{lemma}
\begin{proof}
As $K$ vanishes nowhere on $\widetilde{\Sigma}$, by Theorem 3.1.1 of \cite{PT}, there exist global isothermal coordinates $x,y$ on $\widetilde{\Sigma}$ whose coordinate curves are the lines of curvature. Moreover, $x$ is periodic on $\widetilde{\Sigma}$ with period $2\pi$, and the two fundamental forms can be expressed in terms of the principal curvature $\kappa>0$ as follows:
\begin{equation}\label{1stff}
I= \frac{1}{\kappa}(dx^2+dy^2),\,\,\,\,\,\,\,\,\,\,II=(dx^2-dy^2).
\end{equation}

For $0\leq a< 2\pi$ and any $b$, define $\ell_{a,b}:=\{(a,y):b<y\}$ and $w(z):=e^{-iz}$. $w(\ell_{a,b})$ are the rays going off to $\infty$ for all $a,b$. We claim that the length $L(\widetilde{\Psi}\circ w(\ell_{a,b}))$ of the curve $\widetilde{\Psi}\circ w(\ell_{a,b})$ on $\widetilde{\Sigma}$ is infinite for any $w(\ell_{a,b})$. Suppose $L(\widetilde{\Psi}\circ w(\ell_{c,d}))$ is finite for some $\ell_{c,d}$. Using (\ref{1stff}), one can compute $L(\widetilde{\Psi}\circ w(\ell_{c,d}))$:
$$L(\widetilde{\Psi}\circ w(\ell_{c,d}))=\int_d^\infty\frac{1}{\sqrt{\kappa(c,y)}}dy.$$ This integral can be finite only if
\begin{equation}\label{princ}
\lim_{y\rightarrow\infty}\kappa({c,y})=\infty.
\end{equation}
 When the Gauss map $G$ maps $\widetilde{\Sigma}$ into $\mathbb S^2$, $G$ expands the length of the curve $\widetilde{\Psi}\circ w(\ell_{c,d})$ by the factor of $\kappa(c,y)$ at the point $\widetilde{\Psi}\circ w(\ell_{c,y}))$. Hence one can compute the length of $G\circ\widetilde{\Psi}\circ w(\ell_{c,d})$:
$$L(G\circ\widetilde{\Psi}\circ w(\ell_{c,d}))=\int_d^\infty\frac{1}{\sqrt{\kappa(c,y)}}\cdot\kappa(c,y)dy=\int_d^\infty\sqrt{\kappa(c,y)}dy.$$
It follows from (\ref{princ}) that
\begin{equation}\label{infinity}L(G\circ\widetilde{\Psi}\circ w(\ell_{c,d}))=\infty.
\end{equation}
But we show that this is a contradiction. Remember that being a conformal covering map,  $G\circ\widetilde{\Psi}\circ\pi_s:\mathbb S^2\setminus(\{n,s\}\cup\pi_s^{-1}(\widetilde{\mathcal{P}}))\rightarrow\mathbb S^2$ extends to a holomorphic map from $\mathbb S^2$ to $\mathbb S^2$. Clearly $\gamma_{c,d}:=\pi_s^{-1}(w(\ell_{c,d}))$ has finite length on $\mathbb S^2$ and $G\circ\widetilde{\Psi}\circ\pi_s(\gamma_{c,d})=G\circ\widetilde{\Psi}\circ w(\ell_{c,d})$. Hence $G\circ\widetilde{\Psi}\circ w(\ell_{c,d})$, being the holomorphic image, should also have  finite length on $\mathbb S^2$, which contradicts (\ref{infinity}). Therefore $L(\widetilde{\Psi}\circ w(\ell_{a,b}))=\infty$ for any $\ell_{a,b}$, as claimed. Similarly, $$L(\widetilde{\Psi}\circ w(\ell^-_{a,b}))=\infty\,\,\,\,{\rm for\,\,\,\, all}\,\,\,\, \ell^-_{a,b}:=\{(a,y):y<b\},$$ where $w(\ell^-_{a,b})$ are the rays approaching $O$ away from $\widetilde{\mathcal{P}}$.

Let $\rho$ be a curve in $\mathbb R^2\setminus({\{O\}\cup\widetilde{\mathcal{P}}})$ ending at a puncture $z_0^*\in{\widetilde{\mathcal{P}}}$. The rectangular coordinates $x_1,x_2,x_3$ of $\mathbb R^3$ are harmonic on $\widetilde{\Sigma}$ and their pull-backs under $\widetilde{\Psi}$, $x_1\circ \widetilde{\Psi},x_2\circ\widetilde{\Psi},x_3\circ\widetilde{\Psi}$, are also harmonic in $\mathbb R^2\setminus({\{O\}\cup\widetilde{\mathcal{P}}})$.  Suppose the length of $\rho$ is finite. Then in a neighborhood of $z_0^*$, $x_1\circ \widetilde{\Psi},x_2\circ\widetilde{\Psi},x_3\circ\widetilde{\Psi}$ are bounded harmonic functions with isolated singularity at $z_0^*$. But this is a removable singularity for these harmonic functions. Hence $\widetilde{\Psi}(z_0^*)$ is a regular point of $\widetilde{\Sigma}$ and so $z_0^*$ cannot be a puncture of $\widetilde{\Psi}$ in $\mathbb R^2\setminus({\{O\}\cup\widetilde{\mathcal{P}}})$.  Therefore $\rho$ must have infinite length.
Thus $\widetilde{\Sigma}$ is complete.
\end{proof}

\begin{lemma}\label{5.5}
Suppose the free boundary minimal annulus $\Sigma$ in a ball of $\mathbb R^3$ is embedded. Then the complete minimal surface $\widetilde{\Sigma}$ that is obtained from $\Sigma$ by applying the spherical reflection infinitely many times is the catenoid.
\end{lemma}
\begin{proof}
Our proof will use the embeddedness of $\Sigma$ to derive the convexity of $\partial\Sigma$, which will give the injectivity of the Gauss map on $\Sigma$ and then on $\widetilde{\Sigma}$, which will finally show $\int_{\widetilde{\Sigma}}K=-4\pi$.

Recall that $K<0$ on $\widetilde{\Sigma}$ and so the Gauss map $G:\widetilde{\Sigma}\rightarrow\mathbb S^2$ is a local diffeomorphism. Being the free boundaries of $\Sigma$, $\gamma_1$ and $ \gamma_2$ are lines of curvature, hence they are locally strictly convex (in the opposite directions) on the boundary sphere. In fact, $\gamma_1$ and $\gamma_2$ are strictly convex because $\Sigma$ is embedded.  As the sum of the fluxes of $\gamma_1$ and $\gamma_2$ vanish, one can find a great circle  $\Gamma_0$ that separates $\gamma_1$ and $ \gamma_2$. Now we claim that $G|_{\gamma_1}$ and $G|_{\gamma_2}$ are both injective.  Suppose there exist $p_1,p_2\in \gamma_1$  such that $G(p_1)=G(p_2)$. Let $\Gamma_{1}$ be the equator whose north pole is $G(p_1)$. Clearly $\Gamma_{1}$ is tangent to $\gamma_1$ at both $p_1$ and $p_2$. But this contradicts the convexity of $\gamma_1$. Hence $G|_{\gamma_1}$ must be injective. Similarly, $G|_{\gamma_2}$ is also injective. Let's now show that $G(\gamma_1)$ and $G(\gamma_2)$ are disjoint. Suppose $G(p_1)=G(p_2)$ for some $p_i\in\gamma_i$. Then there is a great circle $\Gamma_{1,2}$ which is tangent to $\gamma_i$ at $p_i,\,i=1,2$ such that $\gamma_1$ and $\gamma_2$ lie on the same side of $\Gamma_{1,2}$ because they are convex curves.  Then the total flux of $\Sigma$ along $\partial\Sigma$ has a nonzero projection in the north pole direction with $\Gamma_{1,2}$ as the equator, which is a contradiction because the total flux should vanish. Therefore $G(\gamma_1)$ and $G(\gamma_2)$ are disjoint, and together they bound $G(\Sigma)$. As $G$ is a local diffeomorphism, $G$ must also be injective in the interior of $\Sigma$.

Remember that we performed repeated reflections to obtain $\Sigma\subset\widetilde{\Sigma}^2\subset\widetilde{\Sigma}^3\subset\cdots\subset\widetilde{\Sigma}^n\subset\cdots,\,\,\,\widetilde{\Sigma}^n:=\Psi^n(\mathbb R\times(-(n-1)a,(n-1)a))$. Since the Gauss map $G$ is a local diffeomorphism and is injective on $\Sigma$, we can see by induction that $G$ is also injective on $\widetilde{\Sigma}^2,\widetilde{\Sigma}^3,\cdots,\widetilde{\Sigma}^n$ for any $n$. Hence $G$ is injective on $\widetilde{\Sigma}$.
As $\widetilde{\Sigma}$ is complete and $G$ is a conformal map (against the negative orientation on $\mathbb S^2$), $G(\widetilde{\Sigma})$ should cover $\mathbb S^2$ almost everywhere, and it should cover only once. More precisely, $G(\widetilde{\Sigma})=\mathbb S^2\setminus (\{n,s\}\cup G(\widetilde{\Psi}(\widetilde{{\mathcal{P}}})))$, where $\widetilde{{\mathcal{P}}}$ is the set of punctures of $\widetilde{\Psi}$ in $\mathbb R^2\setminus\{O\}$. Now we apply a well-known theorem of Osserman (Theorem 19, p.136 \cite{La}) that if $\widetilde{\Sigma}$ is a complete minimal surface with Euler characteristic $\chi$ and with $r$ ends in $\mathbb R^3$, then
$$\int_{\widetilde{\Sigma}}K\leq2\pi(\chi-r).$$
But the fact that $G(\widetilde{\Sigma})$ covers $\mathbb S^2$ once almost everywhere implies $\int_{\widetilde{\Sigma}}K=-4\pi$. Hence $\chi=0$ and $r=2$. Therefore $\widetilde{{\mathcal{P}}}$ is an empty set and $\widetilde{\Sigma}$ is an annular surface. It is here that we should remark that the catenoid is the only complete minimal surface that has total curvature $-4\pi$ and is conformal to $\mathbb S^2$ with two punctures (Corollary 22, p.140 \cite{La}). It follows that $\widetilde{\Sigma}$ is the catenoid.
\end{proof}
Lemma \ref{5.5} finally gives the following theorem.

\begin{theorem}\label{main}
Every embedded free boundary minimal annulus in a ball of $\mathbb R^3$ is the critical catenoid.
\end{theorem}

\begin{remark}
Fern\'{a}ndez-Hauswirth-Mira  \cite{FHM} have constructed infinitely many immersed free boundary minimal annuli $\Sigma$ in a ball and they have shown that each $\Sigma$ has a complete analytic extension $\widetilde{\Sigma}$ which has infinitely many ends. By Lemma \ref{K}, the repeated reflections of their compact immersed minimal annulus $\Sigma$ will give rise to exactly $\widetilde{\Sigma}$ and the Gaussian curvature vanishes nowhere in $\widetilde{\Sigma}$. Moreover, the ends of $\widetilde{\Sigma}$ will correspond to the punctures of $\widetilde{\Psi}$.
\end{remark}

\end{document}